\documentclass[12pt]{amsart}

\usepackage{psfrag}
\usepackage{graphicx}
\usepackage{pinlabel}
\usepackage{amsmath}
\usepackage{amssymb}
\usepackage{amscd}
\usepackage{picinpar}
\usepackage{times}
\usepackage{pb-diagram}
\usepackage{graphicx}
\usepackage{wrapfig}
\usepackage{xspace}
\usepackage{hyperref}
\usepackage{xcolor}

\newtheorem{theorem}{Theorem}[section]
\newtheorem{lemma}[theorem]{Lemma}
\newtheorem{proposition}[theorem]{Proposition}

\begin{document}

\title{ The Deformation Spaces of Geodesic Triangulations of Flat Tori}

\author{Yanwen Luo, Tianqi Wu, Xiaoping Zhu}
%%    Address of record for the research reported here
\address{Department of Mathematics, Rutgers University, New Brunswick
 NJ, 08817}
 \email{yl1594@rutgers.edu}

\address{Center of Mathematical Sciences and Applications, Harvard University, Cambridge, MA, 02138}
 \email{tianqi@cmsa.fas.harvard.edu}

\address{Department of Mathematics, Rutgers University, New Brunswick
 New Jersey 08817}
 \email{xz349@rutgers.edu}

%%Current address
%%\curraddr{Department of Mathematics,
%%]}
%%\thanks will become a 1st page footnote.
\thanks{Acknowledgement: The author was supported in part by NSF 1737876, NSF 1760471, NSF DMS FRG 1760527 and NSF DMS 1811878.}

\keywords{geodesic triangulations, Tutte's embedding}

\begin{abstract}

We prove that the deformation space of geodesic triangulations of a flat torus is homotopy equivalent to a torus. 
This solves an open problem proposed by Connelly et al. in 1983, in the case of flat tori. A key tool of the proof is a generalization of Tutte's embedding theorem for flat tori. When this paper is under preparation, 
Erickson and Lin proved a similar result, which works for all convex drawings.
\end{abstract}

\maketitle

\section{Introduction}

This paper is a continuation of the previous work \cite{luo2020}, where we proved that the deformation space of geodesic triangulations of a surface with negative curvature is contractible.  The purpose of this paper is to identify the homotopy type of the deformation space of geodesic triangulations of a flat torus. This solves an open question proposed by Connelly et al. in \cite{connelly1983problems}. 
The main result of this paper is 

\begin{theorem}
\label{main}
The deformation space of geodesic triangulations of a flat torus is homotopic equivalent to a torus. 
\end{theorem}

It is conjectured in \cite{connelly1983problems} that the space of geodesic triangulations of a closed orientable surface $S$ with constant curvature deformation retracts to the group of orientation preserving isometries of $S$ homotopic to the identity.   This paper affirms this conjecture in the case of flat tori. The case of hyperbolic surfaces has been proved in \cite{luo2020}. In a very recent work, Erickson-Lin \cite{erickson2021planar} proved independently a generalized version of our Theorem \ref{main} for general graph drawings on a flat torus.

The study of the homotopy types of spaces of geodesic triangulations stemmed from the work by Cairns \cite{cairns1944isotopic}.  A brief history of this problem can be found in \cite{luo2020}. These spaces are closely related to diffeomoprhism groups of surfaces.
Bloch-Connelly-Henderson \cite{bloch1984space} proved that the space of geodesic triangulations of a convex polygon is contractible. The space of geodesic triangulations of a planar polygon is equivalent to the space of simplexwise linear homeomorphisms. Hence, the Bloch-Connelly-Henderson theorem can be viewed as a discrete analogue of Smale's theorem, which states that the diffeomorphism group of the closed $2$-disk fixing the boundary pointwise is contractible. Earle-Eells \cite{earle1969fibre} proved that the group of  orientation-preserving diffeomorphisms of a torus isotopic to the identity is homotopy equivalent to a torus. Theorem \ref{main} can be regarded as a discrete version of this theorem.

Similar to the previous work \cite{luo2020}, our key idea to prove Theorem \ref{main} originates from Tutte's embedding theorem.

\subsection{Set Up and the Main Theorem}
Let $\mathbb{T}^2=\mathbb R^2/\mathbb Z^2=[0,1]^2/\sim$ be the flat torus constructed by gluing the opposite sides of the unit square in $\mathbb{R}^2$.

A topological triangulation of $\mathbb T^2$ can be identified as a homeomorphism $\psi$ from $|\mathcal T|$ to $\mathbb T^2$, where $|\mathcal T|$ is the carrier of a 2-dimensional simplicial complex $\mathcal T=(V,E,F)$ with the vertex set $V$, and the edge set $E$, and the face set $F$. For convenience, we label the vertices as $v_1,...,v_n$ where $n=|V|$ is the number of the vertices. Denote $|E|$ as the number of the edges, 
and $\vec e_{ij}$ as the directed edge from the vertex $i$ to its neighbor $j$,
and
$\vec E=\{\vec e_{ij}:ij\in E\}$ as the set of directed edges, 
and $N(i)$ as the indices of neighbored vertices of $v_i\in V$. 

A  \textit{geodesic triangulation} with the combinatorial type $(\mathcal{T},\psi)$ is an embedding $\varphi$ from the one-skeleton $\mathcal{T}^{(1)}$ to $\mathbb{T}^2$ satisfying that
\begin{enumerate}
	\item [(a)] the restriction $\varphi_{ij}$ of $\varphi$ on each edge $e_{ij}$, identified with a unit interval $[0,1]$, is a geodesic of constant speed, and 
	\item [(b)] $\varphi$ is homotopic to the restriction of $\psi$ on $\mathcal{T}^{(1)}$.
\end{enumerate}
Let  $X = X(\mathbb{T}^2, \mathcal{T}, \psi)$ denote the set of all such geodesic triangulations, which is so-called a \textit{deformation space of geodesic triangulations of $\mathbb T^2$}. This space can be defined for other flat tori in a similar fashion. Perturbing each vertex locally, we can construct a family of  geodesic triangulations from an initial geodesic triangulation. Therefore, the space $X$ is naturally a $2n$-dimensional manifold.

For any geodesic triangulation $\varphi\in X$, we can always translate $\varphi$ on $\mathbb T^2$ to make the image $\varphi(v_1)$ of the first vertex $v_1$ be at the (quotient of the) origin $(0,0)$. By this normalization, we can decompose $X$ as $X=X_0\times\mathbb T^2$, where
$$
X_0=X_0(\mathbb T^2,\mathcal T,\psi)=\{\varphi\in X:\varphi(v_1)=(0,0)\}.
$$
Since there are affine transformations between any two flat tori, and an affine transformation always preserves the geodesic triangulations, Theorem \ref{main}  reduces to the following.

\begin{theorem}
\label{contractible}
Given a topological triangulation $(\mathcal T,\psi)$ of $\mathbb T^2$, the space ${X_0} = {X_0}(\mathbb{T}^2, \mathcal{T},\psi)$ is contractible. 
\end{theorem}

\subsection{Key Tool: Generalized Tutte's Embedding Theorem}
Let $\varphi$ be a map from $\mathcal{T}^{(1)}$ to $\mathbb{T}^2$. Assume $\varphi$ maps every edge in $E$ to a geodesic arc parametrized by constant speed on $\mathbb{T}^2$. A positive assignment
$w\in \mathbb{R}_{+}^{\vec{E}}$ on the set of directed edges is called a \textit{weight} of $\mathcal{T}$. We say $\varphi$ is \textit{$w$-balanced} at $v_i$ if 
$$\sum_{ j\in N(i)} w_{ij}\dot\varphi_{ij} = 0,$$
where $\dot\varphi_{ij}=\dot\varphi_{ij}(0)\in T_{\varphi(v_i)}\mathbb T^2\cong\mathbb R^2$.
Then $\dot\varphi_{ij}$ indicates the direction of the edge $\varphi(e_{ij})$ and $\|\dot\varphi_{ij}\|$ equals to the length of $\varphi(e_{ij})$.
A map $\varphi$ is called \textit{$w$-balanced} if it is $w$-balanced at each vertex in $V$.  We have the following version of Tutte's embedding theorem, which is a special case of Gortler-Gotsman-Thurston's embedding result in \cite{gortler2006discrete} and
Theorem 1.6 in \cite{luo2020}.

\begin{theorem}
\label{embedding1}
Assume
$(\mathcal T,\psi)$ is a topological triangulation of $\mathbb T^2$,
and $\varphi$ is a map from $\mathcal{T}^{(1)}$ to $\mathbb{T}^2$ such that $\varphi$ is homotopic to  $\psi|_{\mathcal T^{(1)}}$ and
the restriction $\varphi_{ij}$ of $\varphi$ on each edge $e_{ij}$ is a geodesic parametrized by constant speed. If $\varphi$ is $w$-balanced for some weight $w\in\mathbb R^{\vec E}_+$, then $\varphi$ is an embedding, or equivalently $\varphi$ is a geodesic triangulation. 
\end{theorem}
To be self-contained, we will give a simple proof for Theorem \ref{embedding1} , which is adapted from Gortler-Gotsman-Thurston's argument in \cite{gortler2006discrete}.

The classical Tutte's embedding theorem \cite{tutte1963draw} states that a straight-line embedding of a simple 3-vertex-connected planar graph can be constructed by fixing an outer face as a convex polygon and solving interior vertices on the condition that each vertex is in the convex hull of its neighbors. Various new proofs of Tutte's embedding theorem have been proposed by Floater \cite{floater2003one}, Gortler-Gotsman-Thurston \cite{gortler2006discrete}, et al. 

Tutte's embedding theorem has been generalized by Colin de Verdi{\`e}re \cite{de1991comment},  Delgado-Friedrichs \cite{delgado2005equilibrium}, and   Hass-Scott \cite{hass2012simplicial} to surfaces with non-positive Gaussian curvatures. They showed that the minimizer of a discrete Dirichlet energy is a geodesic triangulation.
Here the fact that $\varphi$ is a minimizer of a discrete Dirichlet energy means that $\varphi$ is $w$-balanced for some symmetric weight $w$ in $\mathbb{R}^{\vec{E}}_+$ with $w_{ij}=w_{ji}$.  Their result also implies that $X = X(\mathbb{T}^2, \mathcal{T}, \psi)$ is not an empty set for any topological triangulation $(\mathcal{T},\psi)$. 
Recently, Luo-Wu-Zhu \cite{luo2020} proved a new version of Tutte's embedding theorem, for nonsymmetric weights and triangulations of orientable closed surfaces with non-positive Gaussian curvature. 

Gortler-Gotsman-Thurston \cite{gortler2006discrete} generalized Tutte's embedding theorem to flat tori.
In contrast to the case of convex polygons and surfaces of negative curvatures, it is not always possible to construct a geodesic triangulation of $\mathbb{T}^2$ such that it is $w$-balanced with respect to a given non-symmetric weight $w$. See \cite{chambers2020morph} Section 1.1 for a detailed discussion.

\subsection{Outline of the Proof}
Fix a lifting $(\tilde x_i,\tilde y_i)\in\mathbb R^2$ of $\psi(v_i)\in\mathbb T^2$ for each $i=1,...,n$. Then for any $\vec e_{ij}\in\vec E$, there exists a unique lifting $\tilde\psi_{ij}:[0,1]\rightarrow\mathbb R^2$ of $\psi_{ij}=\psi|_{e_{ij}}:[0,1]\rightarrow\mathbb T^2$ such that $\tilde\psi_{ij}(0)=(\tilde x_i,\tilde y_i)$. Then
$$
\tilde\psi_{ij}(1)=(\tilde x_j,\tilde y_j)+(b_{ij}^x,b_{ij}^y)
$$
for some lattice point $(b_{ij}^x,b_{ij}^y)\in\mathbb Z^2$.

A geodesic triangulation $\varphi$ in $X_0$ can be represented by $(x_i,y_i)\in\mathbb R^2$ for $i=1, ... ,n$ with $(x_1,y_1)=(0,0)$.
Under this representation, $\varphi_{ij}:[0,1]\rightarrow\mathbb T^2$ is the quotient of the linear map
$$ 
\varphi_{ij}(t)=t( x_j+b_{ij}^x, y_j+b_{ij}^y)+(1-t)(x_i,y_i),
$$
and the equations for $w$-balanced conditions at all the vertices can be written as 
\begin{align*}
    \sum_{j\in N(i)} w_{ij} (x_j - x_i + b^x_{ij}) =& 0,\\
    \sum_{j\in N(i)} w_{ij} (y_j - y_i + b^y_{ij}) =& 0.
\end{align*}
In a closed matrix form, we can write
\begin{equation}
\label{balance}
A(w) \mathbf{x} = \mathbf{b}(w),
\end{equation}
where the weight matrix $A(w)$ is 

$$	A(w) = 
\begin{pmatrix}
-\sum_{j = 1}^n w_{1j} & w_{12} & w_{13} & \dots &  w_{1n}  \\
w_{21} & -\sum_{j = 1}^n w_{2j}& w_{23} & \dots &  w_{2n}\\
w_{31} & w_{32} & -\sum_{j = 1}^n w_{3j} & \dots &  w_{3n} \\
\vdots & \vdots & \vdots & \ddots & \vdots \\
w_{n1} & w_{n2} & w_{n3} & \dots & -\sum_{j = 1}^n w_{nj}
\end{pmatrix} 
,$$
and 
$$\mathbf{x} =  \begin{pmatrix}
x_1& y_1  \\
x_2& y_2 \\
\vdots & \vdots  \\
x_n & y_n
\end{pmatrix}, \quad  \mathbf{b}(w) = \begin{pmatrix}
-\sum_{j = 1}^n w_{1j}b_{1j}^x & -\sum_{j = 1}^n w_{1j}b_{1j}^y  \\
-\sum_{j = 1}^n w_{2j}b_{2j}^x & -\sum_{j = 1}^n w_{2j}b_{2j}^y  \\
\vdots & \vdots  \\
-\sum_{j = 1}^n w_{nj}b_{nj}^x & -\sum_{j = 1}^n w_{nj}b_{nj}^y 
\end{pmatrix}.
$$
Here we denote $w_{ij}=0$ if $e_{ij}\notin E$.

A weight $w$ in $\mathbb{R}^{\vec{E}}_+$ is called \textit{admissible} if the equation (\ref{balance}) is solvable.
Let $W$ be the space of admissible weights. For any $w\in W$, a solution $\mathbf x$ to the equation (\ref{balance}) uniquely determines the coordinates of the vertices, 
and a $w$-balanced  map $\varphi$ that is homotopic to $\psi|_{\mathcal T^{(1)}}$.
By Theorem \ref{embedding1}, such $\varphi$ is an embedding, and $\varphi\in X$.
Noticing that $A(w)$ is weakly diagonal dominant and the graph $\mathcal T^{(1)}$ is connected, the solution to equation (\ref{balance}) is unique up to a 2-dimensional translation, and is unique if we require $(x_1,y_1)=(0,0)$. 
Define the \textit{Tutte map} as
$$
\Psi: W \to {X}_0
$$
sending an admissible weight $w$ to the unique $w$-balanced geodesic triangulation  in ${X}_0$. The Tutte map is continuous by the continuous dependence of the solutions to the coefficients in a linear system. 

\begin{figure}[h!]
  \label{mvc}
  \includegraphics[width=0.3\linewidth]{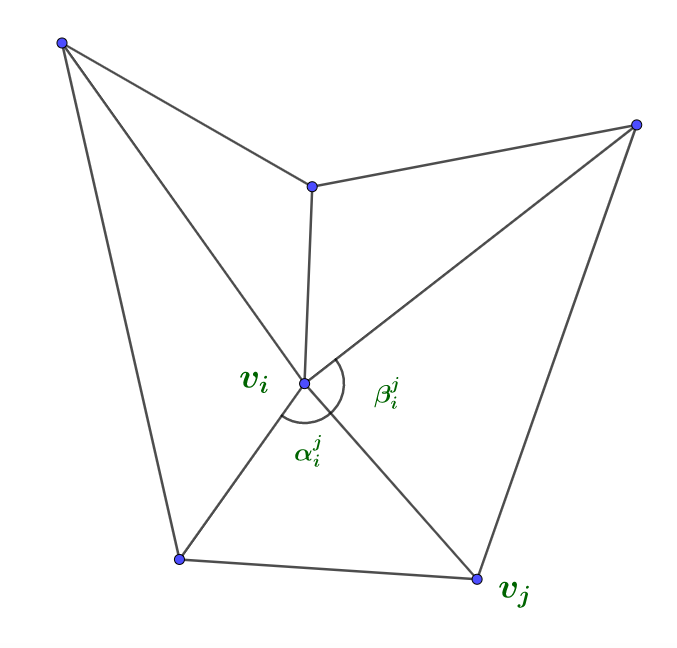}
  \caption{The mean value coordinate.}
\end{figure}

The Tutte map is also surjective, since there exists a smooth map $\sigma$ from ${X}_0$ to $W$ by the ``mean value coordinates"
$$w_{ij}  = \frac{\tan (\alpha_i^j/2) + \tan (\beta_i^j/2)}{l_{ij}}$$
 introduced by Floater \cite{floater2003mean}.
 Here $\alpha_i^j,\beta_i^j$ are two inner angles adjacent to $ e_{ij}$ at the vertex $v_i$, and $l_{ij}$ is the edge length of $\varphi_{ij}$. 
 See Figure \ref{mvc} for an illustration. Floater \cite{floater2003mean} showed that any geodesic triangulation $\varphi$ is $\sigma(\varphi)$-balanced, i.e., $\Phi\circ \sigma=Id_{X_0}$.
 
Having the knowledge of the Tutte map and the mean value coordinates,  Theorem 1.2 reduces to the following proposition.
 \begin{proposition}
\label{retraction}
There exists a continuous map $\Phi: \mathbb{R}^{\vec{E}}_+ \to W$ such that 
$$\Phi|_W = Id_W.$$
\end{proposition}

\begin{proof}[Proof of Theorem \ref{contractible} assuming Proposition \ref{retraction}.]
By Proposition \ref{retraction},  $W$ is contractible since there exists a  retraction from the contractible space $\mathbb{R}^{\vec{E}}_+$ to $W$. So $\sigma\circ\Phi$ is homotopic to the identity map on $W$. On the other hand $\Phi\circ\sigma=Id_{X_0}$, and thus $X_0$ is homotopic to $W$ and contractible.
\end{proof}

\subsection{Organization of this Paper} In Section 2, we will prove Proposition \ref{retraction} by constructing a flow. In Section 3, we prove Theorem \ref{embedding1} following the idea in  Gortler-Gotsman-Thurston's work \cite{gortler2006discrete}. 

\subsection{Acknowledgement} We really appreciate Professor Jeff Erickson for providing valuable insights and references about this problem.

\section{Proof of Proposition \ref{retraction}}
Set an energy function on the weight space $\mathbb R^{\vec E}_+$ as
$$
\mathcal{E}(w) 
= \min_{\mathbf{x}\in \mathbb{R}^{n\times 2}} \big\|A(w)\mathbf{x} - \mathbf{b}(w)\big\|^2
= \min_{\mathbf{x}\in \mathbb{R}^{n\times 2}:x_1=y_1=0 } \big\|A(w)\mathbf{x} - \mathbf{b}(w)\big\|^2,
$$
where the norm is the Frobenius norm of a matrix. 
The second minimization problem above is a least square problem with $2(n-1)$ real variables and a non-degerate coefficient matrix. By the standard formula in linear least squares (LLS) or quadratic programming (QP), the minimizer, denoted as $\mathbf x(w)$, is a smooth function of $w$, and thus $\mathcal E (w)$ is also a smooth function of $w$.
Note that $\mathcal E(w)=0$ if and only if $w$ is admissible, and intuitively $\mathcal E(w)$ measures the deviation of $w$  from being admissible.
The key idea of the proof is to construct a flow on $\mathbb R^{\vec E}_+\backslash W$ to minimize $\mathcal E(w)$, as in the following Lemma \ref{flow1}.
\begin{lemma}
\label{flow1}
There exists a smooth function 
$
\Theta:\mathbb R^{\vec E}_+\backslash W\rightarrow\mathbb R^{\vec E}
$
and a continuous function $C(w)$ on $\mathbb R^{\vec E}_+$
such that for any initial value $w^0\in\mathbb R^{\vec E}_+$, the flow $w(t)$ defined by 
\begin{equation}
\label{ODE}
\Big\{ \begin{array}{ll}
\dot{w}(t) = \Theta(w(t)) \\
 w(0) = w^0
\end{array}
\end{equation}
satisfies that for any $t$ in the maximum existing interval $[0,T)$,
\begin{enumerate}
 \item[(a)] 
 $$
 0 \leq \dot w_{ij}(t) \leq w_{ij}(t),
 $$
 and
 \item[(b)] 
 $$
 \frac{d\mathcal{E}(w(t))}{dt} \leq -C(w^0) \sqrt{\mathcal{E}(w(t))}.
 $$
\end{enumerate}
\end{lemma}
Proposition \ref{retraction} is proved in Section 2.1, assuming this Lemma \ref{flow1}. Then we construct a flow in Section 2.2, and  in Section 2.3 show that this flow is satisfactory for Lemma \ref{flow1}.

\subsection{Proof of Proposition \ref{retraction} Assuming Lemma \ref{flow1}}

Assume $\Theta(w)$ and $C=C(w^0)$ are as in Lemma \ref{flow1}. Given $w^0\in\mathbb R_+^{\vec E}\backslash W$, assume $w(t)$ is the flow defined by equation ({\ref{ODE}}), and 
$[0,T)$ is the maximum existing interval. 

We claim that $T=T(w^0)<\infty$ and $w(t)$ converge to some $\bar w$ as $t\rightarrow T$.
Since 
$$
\frac{d\mathcal E}{dt}\leq -C\sqrt{\mathcal E},
$$ 
we have that 
$$
\frac{d(\sqrt{\mathcal E})}{dt}(t)\leq -\frac{C}{2},
$$
and 
$$
\sqrt{\mathcal{E}(w(t))}\leq \sqrt{\mathcal{E}(w(0))} - \frac{ Ct}{2},
$$
which implies 
$$
T\leq \frac{2\sqrt{\mathcal E(w^0)}}{C(w^0)}<\infty.
$$
Since
$$
 0 \leq \dot w_{ij}(t) \leq w_{ij}(t),
 $$
 we have that
\begin{equation}
\label{ineq}
w_{ij}(t)\leq w^0_{ij}e^{t}\leq w^0_{ij}e^{T}.
\end{equation}
Then by the monotone convergence theorem, $w(t)$ converge to some $\bar w$. By the maximality of $T$, $\bar w$ has to be in $W$. Let $\Phi:\mathbb R^{\vec E}_+\rightarrow W$ be such that $\Phi(w^0)=\bar w$ if $w^0\notin W$, and $\Phi(w)=w$ if $w\in W$.

Now we prove that $\Phi$ is continuous, i.e., for any $w\in \mathbb R^{\vec E}_+$ and $\epsilon>0$, there exists $\delta>0$ such that $|\Phi(w')-\Phi(w)|_\infty\leq\epsilon$ for any $w'$ with $|w'-w|_\infty<\delta$. We consider the two cases $w\in W$ and $w\notin W$.

\subsubsection{$w\in W$}
Since $C(w)$ is continous, there exists $C_1>0$ and $\delta_1>0$ such that $C(w')\geq C_1$ for any $w'$ with $|w-w'|_\infty\leq\delta_1$. Since $\mathcal E$ is continuous, there exists $\delta_2\in(0,\delta_1)$ such that 
$$
\mathcal E(w')\leq 
\left[
\frac{C_1}{2}\log(1+\frac{\epsilon}{2|w|_\infty+\epsilon})
\right]^2
$$
if $|w'-w|_\infty<\delta_2$.
Then we will show that $\delta=\min\{\delta_2,\epsilon/2\}$ is satisfactory. Assume $w'$ satisfies that $|w'-w|_\infty<\delta$. If $w'\in W$, then $|\Phi(w')-\Phi(w)|_\infty=|w'-w|_\infty<\delta\leq\epsilon$. If $w'\notin W$,
$$
T(w')\leq\frac{2\sqrt{\mathcal E(w')}}{C(w')}\leq\frac{C_1\log(1+\epsilon/(2|w|_\infty+\epsilon))}{C_1}
=\log(1+\frac{\epsilon}{2|w|_\infty+\epsilon})\leq\log(1+\frac{\epsilon}{2|w'|_\infty}).
$$
So by the inequality (\ref{ineq}),
$$
|\Phi(w')_{ij}-\Phi(w)_{ij}|\leq|\Phi(w')_{ij}-w'_{ij}|+|w'_{ij}-w_{ij}|
<
w'_{ij}(e^{T(w')}-1)+\frac{\epsilon}{2}\leq\epsilon.
$$
\subsubsection{$w\notin W$}
Assume $\bar w=\Phi(w)\in W$. Then by the result of the previous case, there exists $\delta_1>0$ such that $|\Phi(w')-\Phi(w)|_\infty<\epsilon$ for any $w'$ with $|w'-\bar w|_\infty<\delta_1$. Assume $w(t)$ is the flow determined by the equation (\ref{ODE}) with the initial value $w^0=w$. Then there exists some $t_0$, such that $|w(t_0)-\bar w|_\infty<\delta_1/2$. By the continuous dependence of the solutions of ODEs on the initial values, there exists $\delta>0$ such that if $|w'-w|_\infty<\delta$, then $|w'(t_0)-w(t_0)|<\delta_1/2$, where $w'(t)$ is the flow determined by equation (\ref{ODE}) with the initial value $w^0=w'$. So if $|w-w'|_\infty<\delta$, we have $|w'(t_0)-\bar w|<\delta_1$ and
$$
|\Phi(w')-\Phi(w)|_\infty =|\Phi(w'(t_0))-\Phi(w)|_\infty<\epsilon.
$$

\subsection{Construction of the Flow}
Denote $\mathbf x(w)$ as the minimizer of the second minimization problem in the definition of the energy function $\mathcal E(w)$.
Define  the residual $\mathbf{r}(w)$ as 
$$\mathbf{r}(w) = A(w)\mathbf{x}(w) - \mathbf{b}(w),$$ 
where 
$$\mathbf{r}(w) = \begin{pmatrix}
r_1^x (w)& r_1^y(w)  \\
r_2^x (w)& r_2^y(w) \\
\vdots & \vdots  \\
r_n^x (w)& r_n^y(w)
\end{pmatrix} = \begin{pmatrix}
\mathbf{r}^x (w)& \mathbf{r}^y (w) \\
\end{pmatrix} = \begin{pmatrix}
\mathbf{r}_1 (w)  \\
\mathbf{r}_2 (w)\\
\vdots  \\
\mathbf{r}_n (w)
\end{pmatrix} \in \mathbb{R}^{n \times 2}.$$
The vector $\mathbf{r}_i$ is the residual at the vertex $v_i \in V$, and $\mathbf{r}^x$, $\mathbf{r}^y$ are the projections of the total residual in the directions of the $x$-axis and the $y$-axis respectively. Then by the minimality of $\mathbf{x}(w)$, 
$$A^{T}(w)\mathbf{r}^{x}(w) = 0 \quad \textrm{and} \quad A^{T}(w)\mathbf{r}^{y}(w) = 0.$$
Equivalently, 
$$\mathbf{r}^x(w) \perp A(w)(\mathbb{R}^n), \quad \textrm{and} \quad \mathbf{r}^y(w) \perp A(w)(\mathbb{R}^n).$$
Since $rank(A(w)) = n-1$, $\mathbf{r}^x \| \mathbf{r}^y$. So $rank(\mathbf{r}) \leq 1$, and $\mathbf{r}_i \| \mathbf{r}_j$ for any $1\leq i,j\leq n$. Here $u\| v$ means that vectors $u$ and $v$ are parallel, i.e., linearly dependent.

\begin{lemma}
\label{weights}
Assume $\mathbf r\neq0$ and the following properties holds for $\mathbf{r}_1, \mathbf{r}_2, \cdots, \mathbf{r}_n$.
\begin{enumerate}
	\item [(a)] The vectors have the same direction, namely,
	$$\langle \mathbf{r}_i,\mathbf{r}_j\rangle > 0,  \forall 1 \leq i, j \leq n.$$
	\item [(b)] If 
			$$
			C = \max_{\vec{e}_{ij} \in \vec{E}} \frac{w_{ij}}{w_{ji}}, 
			$$
			 then 
			$$ \frac{\|\mathbf{r}_i\|_2}{\|\mathbf{r}_j\|_2} \leq C^{n-1}, \forall 1 \leq i, j \leq n.$$
\end{enumerate}
\end{lemma}
\begin{proof}
Without loss of generality, after a rotation we can assume that all the vectors $\mathbf{r}_i$ are parallel to $x$-axis, namely $\mathbf{r}^y = {0}$.

To prove part (a), assume that $\langle \mathbf{r}_i,\mathbf{r}_j\rangle = r_i^x\cdot r_j^x \leq 0$ for some $ 1\leq i, j \leq n$. Then one can find a non-zero vector 
$\mathbf{p} = (p_1, \cdots, p_n)^T\in \mathbb{R}^n_{\geq0}$ 
so that $\mathbf{p} \perp \mathbf{r}^x$. Then  $\mathbf{p} \in A(w)(\mathbb{R}^n)$ and there exists $\mathbf{q} =  (q_1, \cdots, q_n)^T\in \mathbb{R}^n$ with $\mathbf{p} = A(w)\mathbf{q}$. Then if $q_i=\max_{j} q_j$ for some $i$, 
$$
0\leq p_i=\sum_{j=1}^nw_{ij}(q_j-q_i)\leq0,
$$
and thus $q_j=q_i$ if $j\in N(i)$. By this maximum principle and the connectedness of the graph, $q_j=q_i$ for any $j\in V$, and $\mathbf p=A(w)\mathbf q=0$. This contradicts with that $\mathbf p$ is nonzero.

If part (b) is not true, ordering the set $V$ based on the values of $r_i^x$ monotonically, one can find a non-empty proper subset $V_0 \subsetneq V$  such that 
$$\frac{\underset{i \in V_0}{\min}  \{r^x_i\}}{\underset{i \in V - V_0}{\max}\{r^x_i\}} > C.$$
Choose a vector $\mathbf{p} \in \mathbb{R}^n$ such that $p_i = 1$ if $i\in V_0$, and $p_i =0$ otherwise. Then the contradiction follows from
\begin{align*}
0 = & \langle \mathbf{r}^x, A(w)\mathbf{p} \rangle = \sum_{i\in V} r_i^x \sum_{j\in N(i)} w_{ij}(p_j - p_i) = \sum_{\vec{e}_{ij}\in \vec{E}} w_{ij}r^x_ip_j - \sum_{i\in V}r^x_ip_i\sum_{j\in N(i)} w_{ij} \\
 = & \sum_{\underset{j\in V_0}{\vec{e}_{ij}\in \vec{E}}} w_{ij}r^x_i - \sum_{i\in V_0}r^x_i \sum_{j\in N(i)} w_{ij} =\sum_{\underset{j\in N(i)}{i\in V_0, j\in V - V_0}}(r_j^xw_{ji} - r_i^x w_{ij}) <0.
\end{align*}
\end{proof}

Assume $\mathbf n=\mathbf n(w)\in\mathbb R^2$ is the unit vector that is parallel to $\mathbf r_1$ and 
$\langle \mathbf n,\mathbf r_1 \rangle>0$.
Define for each directed edge, 
$$u_{ij}=u_{ij}(w) =\mathbf n\cdot(\mathbf x_j - \mathbf x_i + (b_{ij}^x,b_{ij}^y)),  \quad \forall \vec{e}_{ij}\in \vec{E},$$
where $\mathbf x_i=(x_i,y_i)$ is given by the minimizer $\mathbf{x}(w)$. 
Note that  
$$
\|\mathbf r_i\|_2 = \mathbf n\cdot \mathbf r_i=\sum_{j\in N(i)}w_{ij}u_{ij}.
$$

\begin{lemma}
\label{loop}
There exists a constant $\beta=\beta(T,\psi)>0$ such that for any $w\in\mathbb R^{\vec E}_+$, there exists $\vec e_{ij}\in\vec E$ such that $u_{ij}(w)\leq -\beta$.
\end{lemma}
\begin{proof}
Since $u_{ij}=-u_{ji}$ for any $ij\in E$, it suffices to find $\vec e_{ij}$ such that $|u_{ij}(w)|\geq\beta$.
Assume $\mathbf n=(n_1,n_2)\in\mathbb R^2$. Then $|n_1|\geq1/\sqrt2$ or $|n_2|\geq1\sqrt2$.

If $|n_1|\geq1/\sqrt2$, let $\gamma_1=\mathbb T\times\{0\}$ be a horizontal simple loop in $\mathbb T^2$. Then $\psi^{-1}(\gamma_1)$ is a simple loop in the carrier of $\mathcal T$, and it is not difficult to show that there exists a sequence of vertices 
$v(1),...,v(k)=v(0)$ such that $v(i)\sim v(i+1)$ for any $i=0,...,k-1$, and the union $\cup_{i=0}^{k-1}e_{v(i)v(i+1)}$ is a piecewise linear loop in $|\mathcal T|$, which is homotopy equivalent to $\psi^{-1}(\gamma_1)$. By choosing an appropriate orientation, we have that
$$
\sum_{i=0}^{k-1}(\mathbf x_{v(i)+1}-\mathbf x_{v(i)}+(b^x_{v(i+1)v(i)},b^y_{v(i+1)v(i)}))=(1,0).
$$
So
$$
\sum_{i=0}^{k-1}u_{v(i)v(i+1)}=\mathbf n\cdot(1,0)=n_1,
$$
and there exists some $i$ such that $|u_{v(i)v(i+1)}|\geq |n_1|/k\geq1/(\sqrt{2}k)$. Notice that here $k$ is a constant depends only on $\mathcal T$ and $\psi$.

Similarly, if $|n_2|\geq1/\sqrt2$, there exists some $\vec e_{ij}\in \vec E$ 
such that $|u_{ij}|\geq1/(\sqrt{2} k')$ 
for some constant $k'=k'(\mathcal T,\psi)$.
\end{proof}

We define the smooth flow $\Theta$ on the domain $\mathbb R_+^{\vec E}\backslash W$ on each edge as follows
\begin{equation}
\label{flow2}
\Big\{ \begin{array}{ll}
\dot{w}_{ij} = w_{ij}\cdot g\big(\frac{1}{\alpha}(w_{ij} + w_{ji})u_{ij}\big) \cdot h(w_{ij} - w_{ji}), \\
 w_{ij}(0) = w_{ij}^0,
\end{array}
\end{equation}
where $g$ and $h$ are smooth non-increasing functions such that
\begin{enumerate}
	\item [(a)] The function $g \equiv 1$ on $(-\infty, -1)$ and $g \equiv 0$ on $[0, +\infty)$, and
	\item [(b)] The function $h \equiv 1$ on $(-\infty, 1)$, and $h \equiv 0$ on $[2, +\infty)$, and
	\item [(c)] 
	$$
	\alpha=\alpha(w)=\beta\cdot\bigg(2|E|+\sum_{\vec e_{ij}\in \vec E}w_{ij}^{-1}\bigg)^{-1}.
	$$
\end{enumerate}
Roughly speaking, the function $g$ tends be positive if $u_{ij} < 0$, meaning that $w_{ij}u_{ij}$ will decrease so as to reduce the residual $\|\mathbf{r}_i\|_2$. The function $h$ controls the difference between $w_{ij}$ and $w_{ji}$. $\alpha(w)$ is smooth and very small. Specifically we have that
\begin{equation}
\label{alpha}
    \alpha(w)\leq\frac{\beta}{2|E|}\quad\text{ and }\quad\alpha(w)\leq\beta L(w),
\end{equation}
where $L(w)=\min_{\vec e_{ij}\in \vec E}w_{ij}$ is a continuous function on $w\in\mathbb R^{\vec E}_+$. 
Denote 
$$
M(w)=\max\{2,\max_{\vec e_{ij}\in\vec E}|w_{ij}-w_{ji}|\},
$$
which is another continous function on $\mathbb R^{\vec E}_+$.
\begin{lemma}
\label{flow}
Assume the flow $w(t)$ satisfies the equation (\ref{flow2}). Then we have the following.
\begin{enumerate}
	\item [(a)]$0 \leq \dot{w}_{ij} \leq w_{ij}$.
	\item [(b)]$u_{ij} \geq 0$ implies $\dot{w}_{ij} = 0$. Then $\dot{w}_{ij}u_{ij} \leq 0$ for all directed edges.
	\item [(c)]$w_{ij} - w_{ji} \geq 2$ implies $\dot{w}_{ij} = 0$.
	\item [(d)] $L(w(t))$ is non-decreasing and $M(w(t))$ is non-increasing.
	\item [(e)]For any edge $ij$, 
	$$ 
	\frac{w_{ij}(t)}{w_{ji}(t)} \leq 1 + \frac{M}{L}.
	$$

	\item [(f)]The residual vectors $\mathbf{r}_i(t)$ satisfies
	$$ \frac{\max_{i\in V}\|\mathbf r_i(t)\|_2}{\min_{i\in V}\|\mathbf r_i(t)\|_2} \leq (1 + \frac{M}{L})^{n-1}, \quad \forall 1 \leq i, j \leq n.$$
	\item [(g)]The residual vectors $\mathbf{r}_i(t) $ satisfies 
	\begin{equation}
	\label{Er}
	\frac{\sqrt{\mathcal{E}(w(t))}}{\sqrt{n}(1 + M/L)^{n-1}} \leq  \mathbf \|\mathbf r_i(t)\|_2, \quad \forall 1\leq i\leq n.
	\end{equation}
\end{enumerate}
\end{lemma}
\begin{proof}
	Part (a) - (d) are straightforward from equation (\ref{flow2}) and the defining properties of smooth functions $g$ and $h$. 
	Part (e) follows from
	$$
	\frac{w_{ij}(t)}{w_{ji}(t)}
	= 
	 1+ \frac{w_{ij}(t)-w_{ji}(t)}{w_{ji}(t)}  
	\leq
	1+\frac{M}{L}.
	$$
	Part (f) follows from Part (e) and Lemma \ref{weights}. 
	For Part (g), by definition $$\mathcal{E}(t) = \sum_{j=1}^n\|\mathbf{r}_j(t)\|_2^2.$$
	Part (f) implies that 
	$$\mathcal{E}(t) \leq n(1 + \frac{M}{L})^{2n-2}\|\mathbf r_i(t)\|_2^2 
	\quad \text{and} \quad \frac{\sqrt{\mathcal{E}(t)}}{\sqrt{n}(1 + M/L)^{n-1}} \leq \|\mathbf r_i(t)\|_2, \quad\forall 1\leq i\leq n.
	$$
\end{proof}

\subsection{Proof of Lemma \ref{flow1}}
\begin{proof}
Let
$$
C(w)=
\frac{\beta L/M}{2\sqrt{n}(1+M/L)^{n-1}}
$$
where $L=L(w)$ and $M=M(w)$ are continous functions on $\mathbb R^{\vec E}_+$, on which $C(w)$ is also continuous.
We claim that such function $C(w)$ and the flow $\Theta$ defined as equation (\ref{flow2}) are satisfactory. Assume $w^0\in\mathbb R^{\vec E}_+\backslash W$ and $w(t)$ is a flow defined by equation (\ref{flow2}). 
By part (a) of Lemma \ref{flow} we only need to prove the Part (b) of Lemma \ref{flow1}.
By part (d) of Lemma \ref{flow}, it is easy to see that $C(w(t))$ is non-decreasing on $t$. So
we only need to prove that
$$
\frac{d\mathcal E(w(t))}{dt}\leq -C(w(t))\sqrt{\mathcal E(w(t))}.
$$

Given $w\in R^{\vec E}_+$ and $\mathbf x\in\mathbb R^{n\times2}$, denote 
$$
\tilde{\mathcal E}(w,\mathbf x)=\|A(w)\mathbf x-\mathbf b(w)\|^2,
$$
and then
\begin{align*}
	\frac{d{\mathcal{E}}(w(\cdot))}{dt}(t) =& \lim_{\epsilon\to 0}\frac{{\mathcal{E}}(w(t+\epsilon)) - {\mathcal{E}}(w(t))}{\epsilon} \\
	\leq& \lim_{\epsilon\to 0}\frac{\tilde{\mathcal{E}}(w(t+\epsilon), \mathbf{x}(w(t))) - \tilde{\mathcal{E}}(w(t), \mathbf{x}(w(t)))}{\epsilon}  \\
	= &\frac{\partial\tilde{\mathcal{E}}}{\partial w}(w(t), \mathbf{x}(w(t))) \cdot \dot{w}.
\end{align*}
So it suffices to show 
$$\frac{\partial\tilde{\mathcal E}}{\partial w}(w(t), \mathbf{x}(w(t))) \cdot \dot{w} \leq - C\sqrt{\mathcal E(w(t))}.$$
Notice that 
\begin{align*}
		&\frac{\partial\tilde{\mathcal{E}}}{\partial w_{ij}}(w(t),\mathbf x(w(t)))\\
		=&\left(\frac{\partial }{\partial w_{ij} }\sum_{i=1}^n
		\big\|\sum_{j\in N(i)}w_{ij}(\mathbf x_j-\mathbf x_i+(b_{ij}^x,b_{ij}^y))\big\|_2^2\right)\bigg|_{(w,\mathbf x(w))}
		\\ 
		=&  2\mathbf r_i\cdot(\mathbf x_j-\mathbf x_i+(b_{ij}^x,b_{ij}^y))\\
		=& 2\|\mathbf r_i\|_2\mathbf n\cdot(\mathbf x_j-\mathbf x_i+(b_{ij}^x,b_{ij}^y))\\ 
		=&2\|\mathbf r_i\|_2\cdot u_{ij}
\end{align*} 
and then, 
\begin{equation}
\label{1}
\frac{\partial\tilde{\mathcal{E}}}{\partial w}(w(t), \mathbf{x}(w(t))) \cdot \dot{w}  
 = 
 2\sum_{\vec e_{ij}\in\vec E}\|\mathbf r_i\|_2u_{ij} \dot w_{ij}
 \leq
 \frac{2\sqrt{\mathcal{E}(w(t))}}{\sqrt{n}(1 + M/L)^{n-1}}
 \sum_{\vec e_{ij}\in \vec E}u_{ij}\dot w_{ij}.
\end{equation}
Here we use the fact that $\dot{w}_{ij}u_{ij} \leq 0$ for all directed edges (part (b) of Lemma \ref{flow}), and the inequality (\ref{Er}).
It remains to show that
$$
\sum_{\vec e_{ij}\in \vec E}u_{ij}\dot w_{ij}\leq-\frac{\beta L}{2M}.
$$
By Lemma \ref{loop} there exists a directed edge $\vec{e}_{i'j'}$ with $u_{i'j'} \leq -\beta$. 
Then we will consider the following two cases:

\subsubsection{Case 1: $w_{i'j'} - w_{j'i'} \leq 1$}
By the definition of the function $h$, 
$$
h(w_{i'j'}-w_{j'i'})=1.
$$
We also have
$$
g\left(\frac{1}{\alpha}(w_{i'j'}+w_{j'i'})u_{i'j'}\right)=1,
$$
since
$$
(w_{i'j'} + w_{j'i'})u_{i'j'}\leq-2\beta L \leq -\alpha.
$$
By equation ({\ref{flow2}}),
$\dot{w}_{i'j'} = w_{i'j'}$.
Notice that $\dot w_{ij}u_{ij}\leq0$ for any $\vec e_{ij}\in \vec E$ by Part (b) of Lemma \ref{flow}, so
$$
\sum_{\vec e_{ij}\in \vec E}u_{ij}\dot w_{ij}\leq u_{i'j'}\dot w_{i'j'}=u_{i'j'}w_{i'j'}\leq -\beta L\leq -\frac{\beta L}{2M}.
$$
\subsubsection{Case 2: $w_{i'j'} - w_{j'i'} \geq 1$}
Denote
$$
\vec E_0 = \Big\{\vec {e}_{ij} \in \vec{E} \big| u_{ij}<0,(w_{ij} - w_{ji})u_{ij} \geq \alpha \Big\}.
$$ 
If $\vec{e}_{ij}\in\vec E_{0}$, 
then obviously $w_{ij}-w_{ji}<0$ and
$$
h(w_{ij}-w_{ji})=1.
$$
Also,
$$
g\left(\frac{1}{\alpha}(w_{ij}+w_{ji})u_{ij}\right)=1
$$
since
$$
(w_{ij} + w_{ji})u_{ij} \leq (w_{ji} - w_{ij})u_{ij} \leq -\alpha.
$$ 
By equation (\ref{flow2}), $\dot{w}_{ij} = w_{ij}$ and

\begin{align}
\label{2}
	\sum_{\vec e_{ij}\in \vec E}\dot w_{ij}u_{ij} \leq \sum_{\vec {e}_{ij} \in \vec E_0} \dot{w}_{ij} u_{ij}  = \sum_{\vec {e}_{ij} \in \vec E_0} {w}_{ij} u_{ij} \underset{*}{\leq} -\frac{L}{M}\sum_{{e}_{ij} \in E_0} ({w}_{ij} - {w}_{ji})u_{ij}.
\end{align}
The last step (inequality $(*)$) uses the fact that
$
w_{ij}\geq-L(w_{ij}-w_{ji})/M
$,
which is equivalent to that $w_{ji}/w_{ij}\leq1+M/L$.

By the fact that $u_{i'j'}\leq-\beta$, and the assumption $w_{i'j'}-w_{j'i'}\geq1$, 
$$
(w_{i'j'}-w_{j'i'})u_{i'j'}\leq-\beta<0<\alpha,
$$
and thus $\vec e_{i'j'}\notin \vec E_0$. Notice that 
$$
\sum_{\vec e_{ij}\in\vec E}w_{ij}u_{ij}
=\sum_{i=1}^n\sum_{j\in N(i)}w_{ij}u_{ij}=\sum_{i=1}^n\|\mathbf r_i\|_2\geq0,
$$
and
\begin{align*}
	&\sum_{\vec e_{ij}\in\vec E}w_{ij}u_{ij}
	=\sum_{\vec e_{ij}\in\vec E:u_{ij}<0}(w_{ij}-w_{ji})u_{ij}
	\\
	 = & \sum_{\vec{e}_{ij} \in {\vec E_0}} ({w}_{ij} - {w}_{ji})u_{ij} +\sum_{{e}_{ij} \in {\vec E} - \vec E_0 -\{{e}_{i'j'}\}:u_{ij}<0} ({w}_{ij} - {w}_{ji})u_{ij} 
	 + (w_{i'j'} - w_{j'i'})u_{i'j'} \\
	 \leq & \sum_{\vec{e}_{ij} \in {\vec E_0}} ({w}_{ij} - {w}_{ji})u_{ij} 
	 + |E|\alpha  - \beta.
\end{align*}
Then
$$
\sum_{\vec{e}_{ij} \in {\vec E_0}} ({w}_{ij} - {w}_{ji})u_{ij}\geq\beta-|E|\alpha\geq\beta/2,
$$
and
$$
\sum_{\vec e_{ij}\in \vec E}u_{ij}\dot w_{ij}\leq-\frac{\beta L}{2M}.
$$
\end{proof}

\section{Proof of Theorem 1.3}
We will first introduce the concept of discrete one forms, and the Index Theorem proposed by Gortler-Gotsman-Thurston \cite{gortler2006discrete}.

\subsection{Discrete One Forms and the Index Theorem}
A \textit{discrete one form} is a real-valued function $\eta$ on the set of directed edges such that it is anti-symmetric on each undirected edge. Specifically, let $\eta_{ij} = \eta(\vec{e}_{ij})$ be the value of $\eta$ on the directed edge from $v_i$ to $v_j$, then we have
$ \eta_{ij} = - \eta_{ji}.$

For a discrete one form, an edge is degenerate (\textit{resp.} non-vanishing) if the one form is zero (\textit{resp.} non-zero) on it. A vertex is degenerate (\textit{resp.} non-vanishing) if all of edges connected to it are degenerate (\textit{resp.} non-vanishing). A face is degenerate (\textit{resp.} non-vanishing) if all of its three edges are degenerate (\textit{resp.} non-vanishing).  A one-form is degenerate (\textit{resp.} non-vanishing) if all the edges are degenerate (\textit{resp.} non-vanishing). Each edge is either degenerate or non-vanishing. However, vertices or faces can be degenerate, non-degenerate but vanishing on some edges,  or non-vanishing. 

\begin{figure}[h!]
  \includegraphics[width=0.9\linewidth]{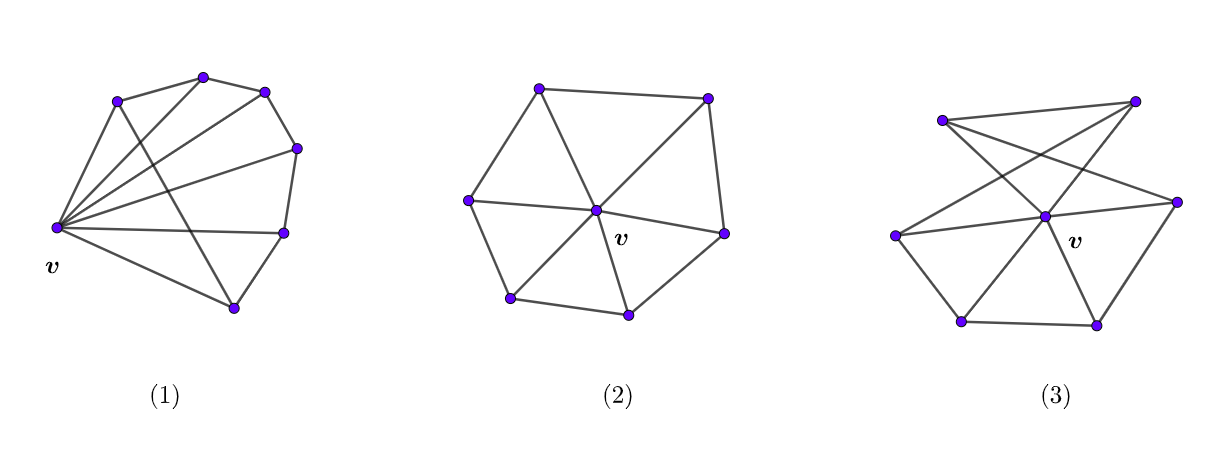}
  \caption{Typical vertex with (1) positive, (2) zero, and (3) negative index.}
  \label{indexcases}
\end{figure}

Assumet $\eta$ is a discrete one form. Deonote $sc(\eta, v)$ as the number of \textit{sign changes} of the non-zero values of $\eta$ on the directed edges starting from $v$, counted in counter-clockwise order. For a vertex $v \in V$, define the \textit{index} of $v$ as  $Ind(\eta, v) = (2-sc(\eta, v))/2$.   Similarly, for a non-degenerate face $t\in F$, the index of $t$ is  $Ind(\eta, t) = (2-sc(\eta,t))/2$, where $sc(\eta, t)$ is the number of sign changes of the non-zero values of $\eta$ on the three edges of $t$, counted in counter-clockwise order. 

The following theorem is a special case of the Index Theorem from \cite{gortler2006discrete}, which is a discrete version of the Poincare-Hopf theorem for discrete one forms. 
\begin{theorem}
\label{index}
Let $\eta$ be a non-vanishing discrete one form on a triangulation of a torus. Then 
$$\sum_{v_i\in V}Ind(\eta, v_i) + \sum_{t_{ijk} \in F} Ind(\eta, t_{ijk}) = 0.$$
\end{theorem}

Assume $\varphi$ satisfies the assumption in Theorem \ref{embedding1}, then for any unit vector $\mathbf n\in\mathbb R^2$
we can naturally construct a discrete one form $\eta$, by letting $\eta_{ij}=\dot\varphi_{ij}\cdot\mathbf n$.
If $\varphi\in X$, a generic unit vector determines a non-vanishing discrete one form $\eta$. Further, if $\varphi\in X$ and
such constructed $\eta$ is non-vanishing, it is not difficult to show that all the indices of the vertices and faces are zero.
Figure 2 illustrates how the neighborhood of $v$ looks like if it has positive, or zero, or negative index, for the case $\mathbf n=(1,0)$.

Based on this construction, we have

\begin{lemma}
\label{perturbation}
Given a triangulation $(\mathcal{T},\psi)$ of $\mathbb{T}^2$, denote $t_{ijk}\in F$ as the triangle with three vertices $v_i$, $v_j$, and $v_k$. There exists a non-vanishing discrete one form $\eta$ such that $\eta_{ij}>0$ and $\eta_{jk}>0$. Moreover, all the indices of the vertices and faces of $\eta$ are zero. 
\end{lemma}

\begin{proof}
By the result of Colin de Verdi{\`e}re \cite{de1991comment} and Hass-Scott \cite{hass2012simplicial}, the space $X(\mathcal T,\psi)$ is not empty for any $(\mathcal T,\psi)$. Let $\varphi$ be a geodesic triangulation in $X$. Then it is not difficult to find a unit vector $\mathbf n$ such that $\dot\varphi_{ij}\cdot\mathbf n>0$ and $\dot\varphi_{jk}\cdot\mathbf n>0$. Define the discrete one form $\eta$ as $\eta_{ij}=\dot\varphi_{ij}\cdot\mathbf n$.
We can perturb the unit vector $\mathbf n$ a little bit to make $\eta$ non-vanishing, and then such $\eta$ is satisfactory.
\end{proof}

\subsection{The Proof of Theorem \ref{embedding1}}
Assume $\varphi: \mathcal{T}^{(1)} \to \mathbb{T}^2$ satisfies the assumption of Theorem \ref{embedding1}, then there exists a unique extension $\bar{\varphi}: |\mathcal{T}| \to \mathbb{T}^2$ such that the restriction of $\bar{\varphi}$ to every face is linear. Such $\bar\varphi$ is homotopic to $\psi$, and $\varphi$ is a geodesic triangulation in $X$ if and only if $\bar{\varphi}$ is a homeomorphism. 

For any triangle $t_{ijk}\in X$, we say that $\bar\varphi(t_{ijk})$ is \textit{degenerate} if $\bar\varphi(t_{ijk})$ is contained in some geodesic $\lambda$. If $\bar\varphi(t_{ijk})$ is not degenerate, we can naturally define its inner angle $\theta^i_{jk}$ at  $\varphi(v_i)$. 
We claim that 

\begin{enumerate}
	\item [(a)] $\bar\varphi(t_{ijk})$ is not degenerate for any $t_{ijk}\in F$, and
	\item [(b)] $\bar\varphi$ is locally a homeomorphism.
\end{enumerate}
Then $\bar\varphi$ is a proper local homeomorphism, and thus is a covering map. Since $\bar\varphi$ is homotopic to the homeomorphism $\psi$, $\bar\varphi$ is indeed a degree-one covering map, i.e., a homeomorphism.

\subsubsection{Proof of Claim (a)}

Assume there is some triangle $t \in F$ such that $\bar{\varphi}(t)$ is degenerate and hence contained in a geodesic $\lambda$. Here $\lambda$ is assumed to be a closed geodesic, or a densely immersed complete geodesic.
Let $\mathcal{C}$ be the union of all triangles $t$ such that $\bar\varphi(t)\subset\lambda$.
Then $\mathcal{C}$ is not the whole complex $\mathcal{T}$, otherwise $\bar{\varphi}$ is not homotopic to the homeomorphism $\psi$. 
So, we can find a vertex $v_0 \in \partial \mathcal{C}$. Denote $star(v_0)$ as the star-neighborhood of $v_0$ in $\mathcal T$. Then $\bar\varphi(star(v_0))$ is not in $\lambda$, but $\bar\varphi(t_0)\subset\lambda$ for some triangle $t_0$ in $star(v_0)$.

Let $\mathbf n$ be a unit vector that is orthogonal to the geodesic $\lambda$, and define $\eta$ as $\eta_{ij}=\dot\varphi_{ij}\cdot\mathbf n$.
Then the vertex $v_0$ is non-degenerate with respect to $\eta$, but the face $t_0$ is degenerate. Let $\xi$ be a discrete one form in Lemma \ref{perturbation} with the triangle $t_0$ and $v_j = v_0$. 
Scale $\xi$ to make it very small so that $\eta + \xi$ has the same signs with $\eta$ on the non-degenerate edges of $\eta$.

Notice that $sc(\eta, v_0)\neq 0$, or equivalently $sc(\eta,v_0)\geq2$, otherwise all the edges connecting $v_0$ lie on a half-space, which contradicts to the assumption that $\varphi$ is balanced. 
Since $t_0$ is degenerate in $\eta$, taking the opposite $-\xi$ instead of $\xi$ if necessary, we can assume that $sc(\eta, v_0) < sc(\eta + \xi, v_0)$, and thus $Ind(\eta+\xi,v_0)<0$. 

Noticing that $\eta+\xi$ is non-vanishing, we will derive a contradiction with the Index Theorem \ref{index}, by showing that the index of $\eta+\xi$ is nonpositive for any vertex and face.

If a face $t$ is degenerate in $\eta$, then $Ind(\eta +  \xi, t)  = Ind(\xi, t) = 0$. If a face $t$ is non-degenerate in $\eta$, then $Ind(\eta + \xi, t)=Ind(\eta)=0$. In fact, the index of any non-degenerate face is zero. 

If a vertex $v$ is degenerate in $\eta$, then $Ind(\eta +  \xi, v)  = Ind(\xi, v) = 0$. If a vertex $v$ is non-vanishing, then $Ind(\eta +  \xi, v)  = Ind(\eta, v)$. Since $\varphi$ is balanced, $Ind(\eta, v) \leq 0$. 

If a vertex $v$ is non-degenerate but vanishing at some edges in $\eta$, then 
$$Ind(\eta +  \xi, v)  \leq Ind(\eta, v) \leq 0,$$ 
since adding $\xi$ can only introduce more sign changes. 

\subsubsection{Proof of Claim (b)}

Since $\varphi$ is $w$-balanced, it is not difficult to show that for any vertex $i$, 
$$
\sum_{jk:t_{ijk}\in F}\theta_{jk}^i\geq2\pi,
$$
and the equality holds if and only if all the edges around $v_i$ do not ``fold'' under the map $\bar\varphi$.
By the Gauss-Bonnet theorem, 
$$
\sum_{i=1}^n(2\pi-\sum_{jk:t_{ijk}\in F}\theta^i_{jk})=0.
$$
So 
$$
\sum_{jk:t_{ijk}\in F}\theta_{jk}^i=2\pi,
$$
for any vertex $v_i$, and all the edges in $E$ do not fold. Thus, $\bar\varphi$ is a local homeomorphism.

\bibliography{ref} 
\bibliographystyle{amsplain}

\end{document}